\documentclass{amsart}
\usepackage{amsmath,amssymb,mathrsfs,bbm,yhmath,longtable,mathtools}
\usepackage{amsthm,rotating,xcolor,epsfig,hyperref,url,MnSymbol,rotate}
\usepackage{graphicx}
\usepackage[weather]{ifsym}

\date{}

\theoremstyle{theorem}
\newtheorem{theo}{Theorem}[section]
\newtheorem{lemm}{Lemma}[section]
\newtheorem{prop}{Proposition}[section]
\newtheorem{coro}{Corollary}[section]

\theoremstyle{remark}

\theoremstyle{definition}
\newtheorem{defi}{Definition}[section]

\numberwithin{equation}{section}
\numberwithin{figure}{section}

\graphicspath{{./pictures/}}	

\author {Ivan Dynnikov and Vera Sokolova}
\address{\noindent Steklov Mathematical Institute of Russian Academy of Sciences, 8 Gubkina Str., Moscow 119991, Russia}
\address{\noindent St.\ Petersburg State University, Line 14th (Vasilyevsky Island), 29, Saint Petersburg, 199178, Russia}
\thanks{The work of the first named author is supported by the Russian Science Foundation under grant~19-11-00151.}
\email{dynnikov@mech.math.msu.su}
\address{Lomonosov Moscow State University, 1 Leninskije gory, Moscow 119991, Russia}
\email{sokolova.vera00@yandex.ru}

\title{Multiflypes of rectangular diagrams of links}

\begin{document}

\maketitle
\begin{abstract}
We introduce a new very large family of transformations of rectangular diagrams of links
that preserve the isotopy class of the link.
We provide an example when two diagrams of the same complexity are related by such a transformation and
are not obtained from one another by any sequence of `simpler' moves not increasing the complexity
of the diagram along the way.
\end{abstract}

\section*{Introduction}
It is shown in~\cite{dyn06} that rectangular diagrams of links (also known as
arc-presentations and grid diagrams) allow one
to solve certain decidability problems in knot theory using one of the most
naive approaches, which is based on monotonic simplification.
Namely, one can decide wether the given rectangular diagram
represents an unknot, a split link, or a composite link
by successively applying all possible sequences of elementary moves not increasing
the number of edges, and check if any of the obtained diagrams
is trivial, split, or composite, respectively. Previously known solutions
of these problems, the first of which are due to W.\,Haken~\cite{haken}
and H.\,Schubert~\cite{schubert}, use much more advanced technique.

Elementary moves involved in the monotonic simplification procedure mentioned
above include only very simple transformations called exchange moves,
stabilizations and destabilizations. There are several reasons to look for
more general families of moves preserving the isotopy class of the link.

One reason is that more general moves might make the monotonic
simplification faster. To this writing, the algorithms based on monotonic
simplification of rectangular diagrams have exponential asymptotic complexity
due to the fact that the simplification is not \emph{strictly} monotonic.

Another reason is a hope that more general moves would allow to solve
more algorithmic problems in the same manner. One of such problems,
which is most natural to consider after the unknotedness, splitness,
and factorization ones, is finding the JSJ-decomposition of the link complement
(solved in~\cite{jt,jlr} with the help of Kneser--Haken normal surfaces). It is also
natural to try extending the monotonic simplification approach to general links.

Finally, studying the combinatorics of more general transformations may
result in new classification results and more efficient estimates for the number
of elementary moves (or Reidemeister moves
for planar diagrams) needed to transform one diagram to another if
they represent isotopic links.

A class of transformations of rectangular diagrams generalizing
elementary moves was introduced in~\cite{dy03}, where the new
transformations were called flypes, since in certain situations they
converted into flypes of the respective planar diagrams. However,
these moves did not help to advance in any of the directions listed above.

In particular, it is shown in~\cite{kaza} that flypes of rectangular diagrams
do not allow to detect satellite knots by means of monotonic simplification.
An example of two rectangular diagrams, which we denote here
by~$R_{\text{\Sun}}$ and~$R_{\text{\Cloud}}$ (with the former modified in an obvious way by
exchange moves),
representing the same satellite knot are provided (see~\cite[Section~7]{kaza}) such
that~$R_{\text{\Cloud}}$  is not `obviously satellite' and admits no complexity preserving flype
changing the combinatorial type of the diagram, whereas~$R_{\text{\Sun}}$
is `obviously satellite'.

Below we introduce a much more general type of moves,
which we call multiflypes because they have been originally thought
of as several flypes performed simultaneously. The main result
of the present paper is a proof that these new moves preserve the isotopy
class of the link. We also use the example from~\cite{kaza} to show
an advantage of the new moves: they allow to proceed from~$R_{\text{\Cloud}}$ to~$R_{\text{\Sun}}$
without increasing the complexity along the way.

\section{Preliminaries}

We denote by~$\mathbb T^2$ the two-dimensional torus~$\mathbb S^1\times\mathbb S^1$, and by~$\theta,\varphi$
the angular coordinates on~$\mathbb T^2$, which run through~$\mathbb S^1=\mathbb R/(2\pi\mathbb Z)$. Denote
by~$p_\theta$ and~$p_\varphi$ the projection maps from~$\mathbb T^2$ to the first and the second~$\mathbb S^1$-factors, respectively. For any~$\theta_0,\varphi_0\in\mathbb S^1$, we put~$m_{\theta_0}=\{\theta_0\}\times\mathbb S^1$,
$\ell_{\varphi_0}=\mathbb S^1\times\{\varphi_0\}$, and call these \emph{a meridian} and \emph{a longitude} of~$\mathbb T^2$,
respectively.

For two distinct points~$x_1,x_2\in\mathbb S^1$ we denote by~$[x_1;x_2]$ (respectively, $(x_1;x_2)$) the closed (respectively, open) interval in
$\mathbb S^1$ starting at~$x_1$ and ending at~$x_2$.

\begin{defi}\label{R-def}
\emph{An oriented rectangular diagram of a link} is a non-empty finite subset~$R\subset\mathbb T^2$
with a decomposition~$R=R^+\sqcup R^-$ into a disjoint union of two subsets~$R^+$ and~$R^-$ such
that we have~$p_\theta(R^+)=p_\theta(R^-)$, $p_\varphi(R^+)=p_\varphi(R^-)$, and each of~$p_\theta$, $p_\varphi$
restricted to each of~$R^+$, $R^-$ is injective.

The elements of~$R$ (respectively, of~$R^+$ or $R^-$)
are called \emph{vertices} (respectively, \emph{positive vertices} or \emph{negative vertices}) of~$R$.

Pairs~$(u,v)$ of vertices of~$R$ such that~$p_\theta(u)=p_\theta(v)$ (respectively, $p_\varphi(u)=p_\varphi(v)$)
are called \emph{vertical} (respectively, \emph{horizontal}) \emph{edges} of~$R$.

All points in
$$\bigl(p_\theta(R)\times p_\varphi(R)\bigr)\setminus R\subset\mathbb T^2$$
are called \emph{crossings} of~$R$.
\end{defi}

With every oriented rectangular diagram of a link~$R$ one associates
a topological oriented link type~$\mathscr L(R)$ as follows.
First, choose a meridian~$m_{\theta_0}$ and a longitude~$\ell_{\varphi_0}$ not passing
through a vertex of~$R$ and cut~$\mathbb T^2$ along~$m_{\theta_0}\cup\ell_{\varphi_0}$
to obtain a square. Then connect, by a straight line segment, every pair
of vertices of~$R$ forming an edge of~$R$. At every intersection point, regard
the vertical arc as overcrossing. The vertical arcs are oriented from a positive
vertex to a negative one, and horizontal arcs from a negative vertex to a positive one.
The obtained oriented planar diagram of a link
represents~$\mathscr L(R)$. An example is shown in Figure~\ref{rd-fig}.
\begin{figure}[ht]
\begin{tabular}{ccc}
\includegraphics[width=150pt]{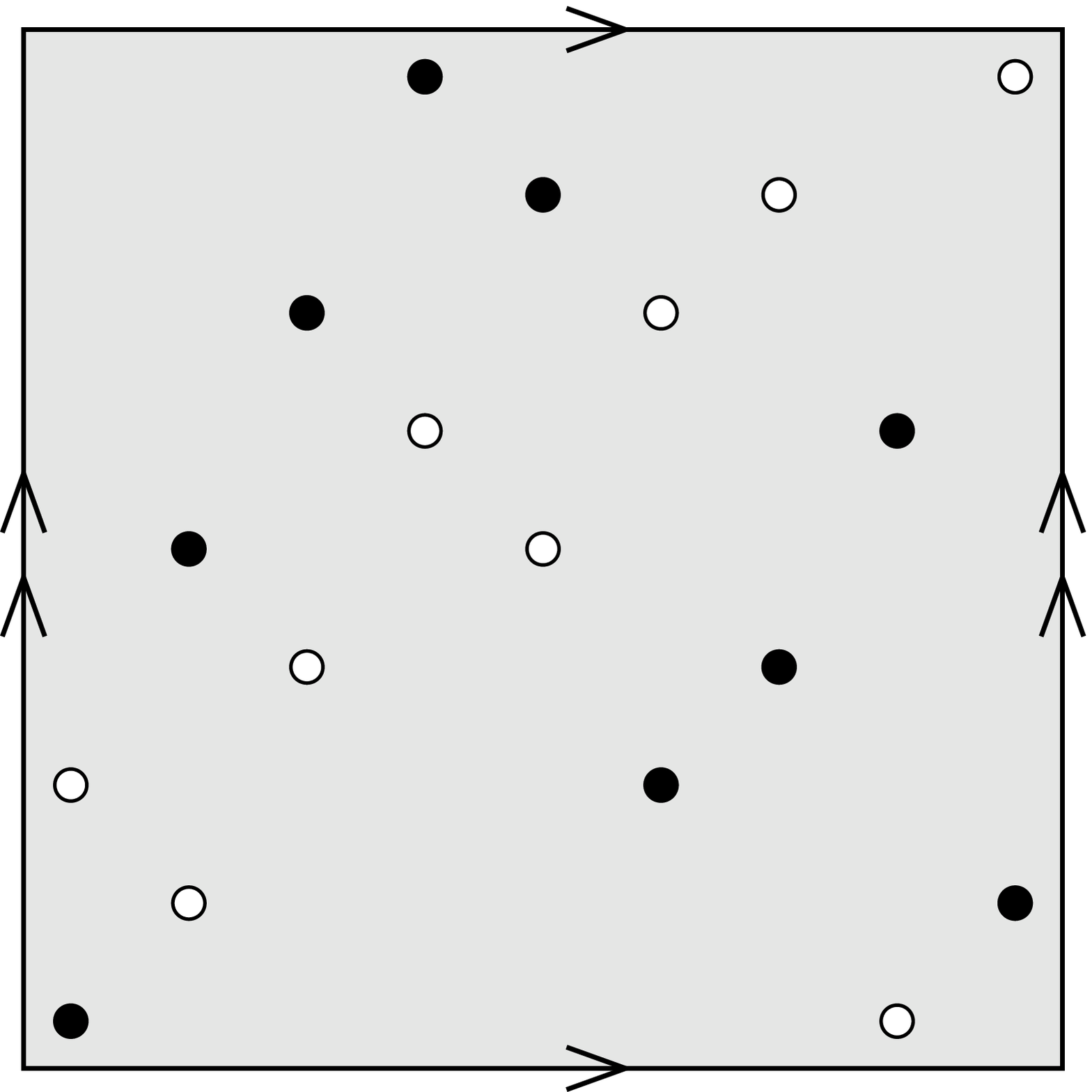}\put(-80,20){$\mathbb T^2$}&\hbox to 2cm{\hss}&\includegraphics[width=150pt]{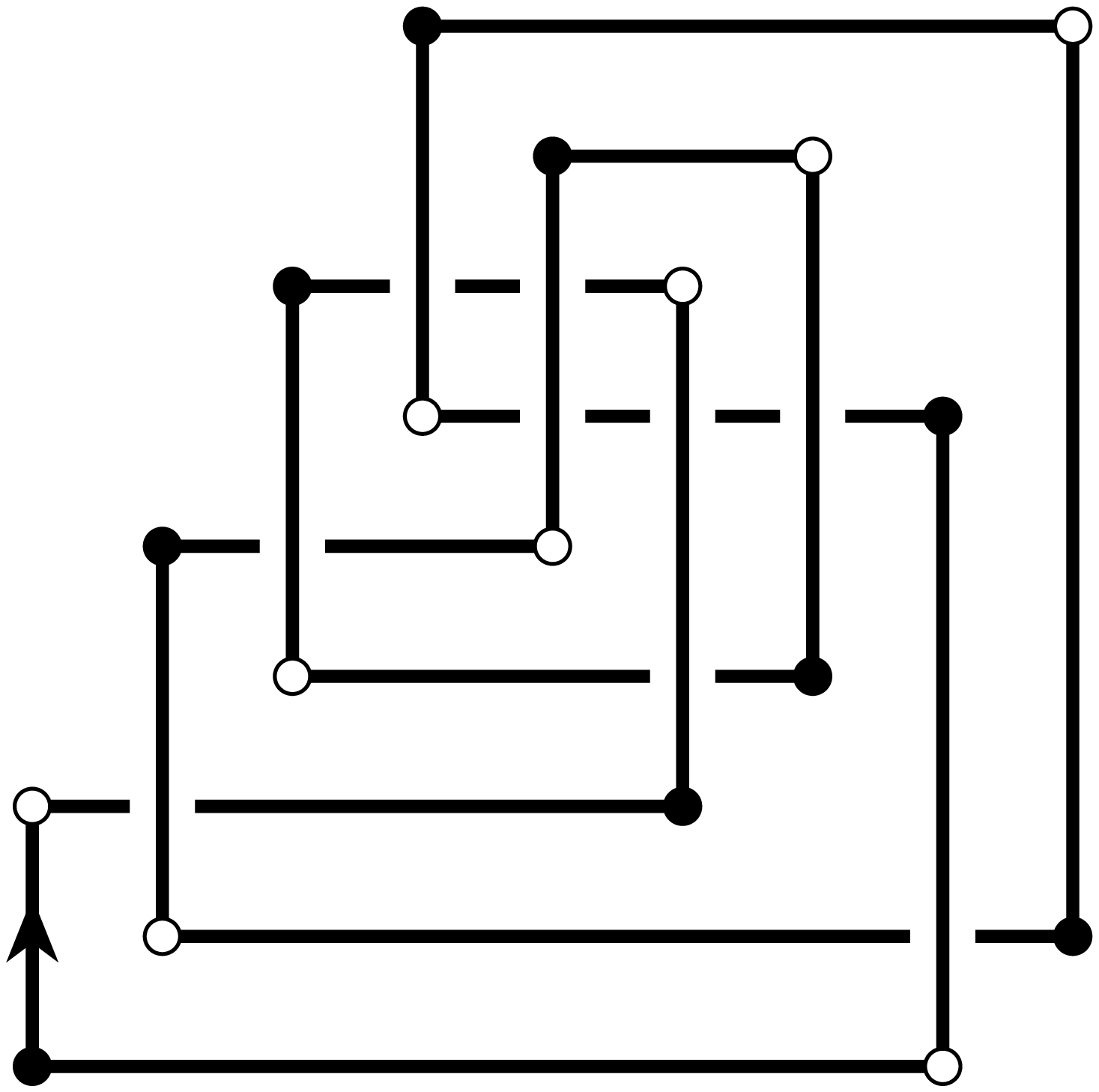}\\
$R$&&a representative of~$\mathscr L(R)$
\end{tabular}
\caption{A rectangular diagram of a link and a planar diagram of the corresponding link}\label{rd-fig}
\end{figure}
Here and below positive vertices are shown in black, and negative vertices in white.

It will be convenient in the sequel to represent any rectangular diagram of a link~$R$
by the following function~$\sigma:\mathbb T^2\rightarrow\{-1,0,1\}$, which will be called
\emph{the characteristic function of~$R$}:
$$\sigma_R(v)=\left\{\begin{aligned}0,&\text{ if }v\notin R,\\
1,&\text{ if }v\in R^+,\\
-1,&\text{ if }v\in R^-.
\end{aligned}\right.$$

By \emph{a rectangle} we mean a subset~$r$ of~$\mathbb T^2$ of the form~$[\theta_1;\theta_2]\times[\varphi_1;\varphi_2]$,
where~$\theta_1,\theta_2,\varphi_1,\varphi_2\in\mathbb S^1$. With every
rectangle~$r=[\theta_1;\theta_2]\times[\varphi_1;\varphi_2]$ we associate \emph{a trivial rectangular
diagram of a link}~$R(r)$ as follows:
$$R(r)^+=\{(\theta_1,\varphi_1),(\theta_2,\varphi_2)\},\quad
R(r)^-=\{(\theta_1,\varphi_2),(\theta_2,\varphi_1)\}.$$
Clearly, $R(r)$ represents an unknot.

For any rectangle~$r$, we denote~$\sigma_{R(r)}$ by~$\sigma_r$ for brevity.

\begin{defi}
Let~$R$ and~$R'$ be oriented rectangular diagrams of links. The passages
from~$R$ to~$R'$ and from~$R'$ to~$R$ are called \emph{elementary moves} if there is a rectangle~$r$ such that:
\begin{enumerate}
\item
$\sigma_R-\sigma_{R'}=\sigma_r$
\item
the intersection~$R\cap r$ consists of exactly one, two, or three successive vertices of~$r$.
\end{enumerate}
\end{defi}

Elementary moves defined in this way include all versions of
exchange moves (also called commutations in the literature), stabilizations
and destabilizations introduced in earlier works~\cite{crom,dyn06},
and also some compositions of these moves with several exchange moves. It is easy to verify that
all elementary moves preserve the isotopy class of the link associated with the diagram.

\section{Definition of a multiflype and the main result}\label{def-sec}
There are four similar versions of multiflypes related with one another by
symmetries~$(\theta,\varphi)\mapsto(-\theta,\varphi)$ and~$(\theta,\varphi)\mapsto(\theta,-\varphi)$.
Each type of multiflypes is assigned an arrow~$\nearrow$, $\nwarrow$, $\swarrow$, or~$\searrow$,
on which the symmetries act accordingly.

Let~$R$ be an oriented rectangular diagram of a link, and let~$A\subset\mathbb T^2$ be
an annulus such that:
\begin{enumerate}
\item
the boundary~$\partial A$ is transverse to all meridians and longitudes, and
the slope of~$\partial A$ is positive, that is, $d\varphi/d\theta>0$ on~$\partial A$;
\item
$\partial A$ misses all crossings of~$R$ (which are defined in Definition~\ref{R-def});
%
%
\item
there is no pair of distinct points~$u,v\in\partial A$ not forming a vertical (respectively,
horizontal) edge of~$R$
but lying on the same meridian (respectively, longitude) and such that~$p_\varphi(u),p_\varphi(v)\in p_\varphi(R)$
(respectively,~$p_\theta(u),p_\theta(v)\in p_\theta(R)$).
\end{enumerate}
%

We denote by~$\partial_1A$ the connected component of~$\partial A$
defined by demanding that a small push off of~$\partial_1A$ in the~$(1,-1)$-direction
lies outside of~$A$. The other connected component is denoted by~$\partial_2A$.

For every point~$v\in A\setminus\partial A$, denote by~$r_v$ a rectangle~$[\theta_1;\theta_2]\times
[\varphi_1;\varphi_2]$ such that~$(\theta_1,\varphi_1)=v$, $(\theta_2,\varphi_1)\in\partial_1A$, and
$(\theta_1,\varphi_2)\in\partial_2A$ (see Figure~\ref{rv-fig}).
\begin{figure}[ht]
\includegraphics[width=150pt]{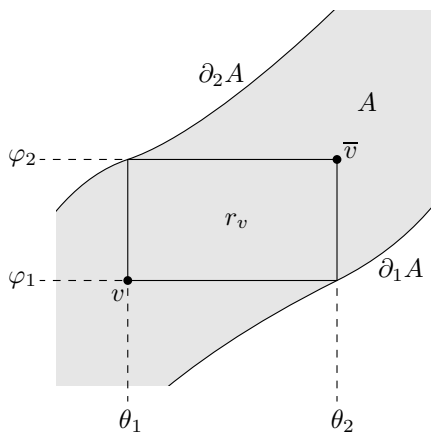}
\put(-30,110){$A$}\put(-22,47){$\partial_1A$}\put(-90,121){$\partial_2A$}
\put(-123,38){$v$}\put(-80,67){$r_v$}\put(-120,-10){$\theta_1$}\put(-40,-10){$\theta_2$}
\put(-162,45){$\varphi_1$}\put(-162,91){$\varphi_2$}\put(-35,92){$\overline v$}
\caption{The rectangle~$r_v$}\label{rv-fig}
\end{figure}
Such a rectangle is clearly unique. Denote by~$\overline v$ the vertex of~$r_v$ opposite to~$v$.

\begin{prop}\label{well-defined-prop}
There exists a \emph(unique\emph) oriented rectangular diagram of a link~$R'$ such
that
\begin{equation}\label{sigma-flype-eq}
\sigma_{R'}=\sigma_R-\sum_{v\in R\cap(A\setminus\partial A)}\sigma_R(v)\sigma_{r_v}.
\end{equation}
\end{prop}

\begin{proof}
We give a geometric interpretation of~\eqref{sigma-flype-eq} from
which it is clear that~$R'$ is a well defined oriented rectangular diagram of a link.

First, note that, on any meridian~$m_{\theta_0}$ and on any longitude~$\ell_{\varphi_0}$,
the right hand side of~\eqref{sigma-flype-eq} sum up to zero, since so does
each summand in it. So, it suffices to verify that the right hand side of~\eqref{sigma-flype-eq}
takes only values in~$\{-1,0,1\}$, and on every meridian and longitude, it takes non-zero
values at at most two points.

The map~$v\mapsto\overline v$ is clearly a bijection from~$A\setminus\partial A$ to itself.
If~$v\in R\cap(A\setminus\partial A)$, then the subtraction of~$\sigma_R(v)\sigma_{r_v}$ from~$\sigma_R$, geometrically,
results in removing~$v$ from the diagram and adding~$\overline v$ with the opposite sign. So,
inside the domain~$A\setminus\partial A$, the geometric meaning of~\eqref{sigma-flype-eq}
is the replacement of every vertex~$v$ in~$R\cap(A\setminus\partial A)$
by the respective vertex~$\overline v$ having the opposite sign.

Some vertices are also removed or added at~$\partial A$, and the rule defined by~\eqref{sigma-flype-eq}
is as follows. Let~$(\theta_0,\varphi_0)\in\partial_1A$, and let~$[\theta_1;\theta_0]\times\{\varphi_0\}$
be a maximal horizontal arc contained in~$A$. If there are two or no vertices of~$R$
in the open arc~$(\theta_1;\theta_0)\times\{\varphi_0\}$, then no change of the diagram occurs at~$(\theta_0,\varphi_0)$.
If this arc contains exactly one vertex and~$(\theta_0,\varphi_0)$ is also a vertex of~$R$,
then this vertex is removed. Otherwise, a vertex is added at~$(\theta_0,\varphi_0)$.

The change of the diagram at any~$(\theta_0,\varphi_0)\in\partial_2A$ depends similarly
on the number of vertices of~$R$ in~$\{\theta_0\}\times(\varphi_1;\varphi_0)$,
where~$\{\theta_0\}\times[\varphi_1;\varphi_0]$ is a maximal vertical arc contained in~$A$.

Thus, the only way in which the right hand side of~\eqref{sigma-flype-eq} may fail to be the characteristic
function of an oriented rectangular diagram is that it takes non-zero values at four or more
points contained in a single meridian or longitude. One can see that the conditions
imposed on the choice of~$A$ guarantee that this does not happen.
\end{proof}

The passage from~$R$ to~$R'$ defined by~\eqref{sigma-flype-eq} is called \emph{a $\nearrow$-multiflype} (\emph{based on~$A$}).
The other types of multiflypes are defined as follows:

\centerline{\begin{tabular}{cl}
$s_\updownline(R)\mapsto s_\updownline(R')$&is a $\nwarrow$-multiflype,\\
$s_\leftrightline(R)\mapsto s_\leftrightline(R')$& is a $\searrow$-multiflype,\\
$(s_\updownline\circ s_\leftrightline)(R)\mapsto(s_\updownline\circ s_\leftrightline)(R')$
& is a $\swarrow$-multiflype,
\end{tabular}}
where
$$s_\updownline(\theta,\varphi)=(-\theta,\varphi),\quad s_\leftrightline(\theta,\varphi)=(\theta,-\varphi).$$

The proof of the following two statements is easy and left to the reader.

\begin{prop}
The inverse of a $\nearrow$-multiflype \emph(respectively, $\nwarrow$-multiflype\emph)
is a $\swarrow$-multiflype \emph(respectively, $\searrow$-multiflype\emph).
\end{prop}

\begin{prop}
Elementary moves of oriented rectangular diagrams of links are exactly multiflypes such that the interior of
the respective annulus~$A$ contains exactly one vertex of the diagram.
\end{prop}

The following theorem is the main result of this paper.

\begin{theo}\label{main-th}
If~$R\mapsto R'$ is a multiflype, then~$\mathscr L(R)=\mathscr L(R')$.
\end{theo}

The proof will be given in Section~\ref{proof-sec}.

\section{An example}
Shown at the top of Figure~\ref{exam-fig} is an oriented rectangular diagram~$R_{\text{\Sun}}$ of a link
which is a satellite knot. Namely, it is a 2-cable of the trefoil knot, and the
satellite structure is clearly visible from the diagram.

The diagram~$R_{\text{\Cloud}}$ in the middle row represents the same knot, but
it is already non-trivial to detect the satellite structure
from this diagram. It is easy to see that
no combinatorially non-trivial and complexity-preserving elementary move
can be applied to~$R_{\text{\Cloud}}$. Moreover, it is shown in~\cite{kaza}
that the combinatorial structure of~$R_{\text{\Cloud}}$ cannot be changed
by more general moves called flypes in~\cite{dy03}, without
introducing more edges. Thus, with only flypes at hand,
the monotonic simplification method does fails at detecting
the satellite structure of this knot from the diagram~$R_{\text{\Cloud}}$.

This detection becomes possible with the help of multiflypes. The diagram~$R_{\text{\SunCloud}}$
in Figure~\ref{exam-fig}, if viewed combinatorially, is obtained
from~$R_{\text{\Cloud}}$ by a single multiflype preserving the number of edges. This is demonstrated in
the bottom row of Figure~\ref{exam-fig}, where the respective annulus~$A$
is shown as a shaded region, and all involved rectangles
of the form~$r_v$ are also indicated.
\begin{figure}[ht]
\includegraphics[width=150pt]{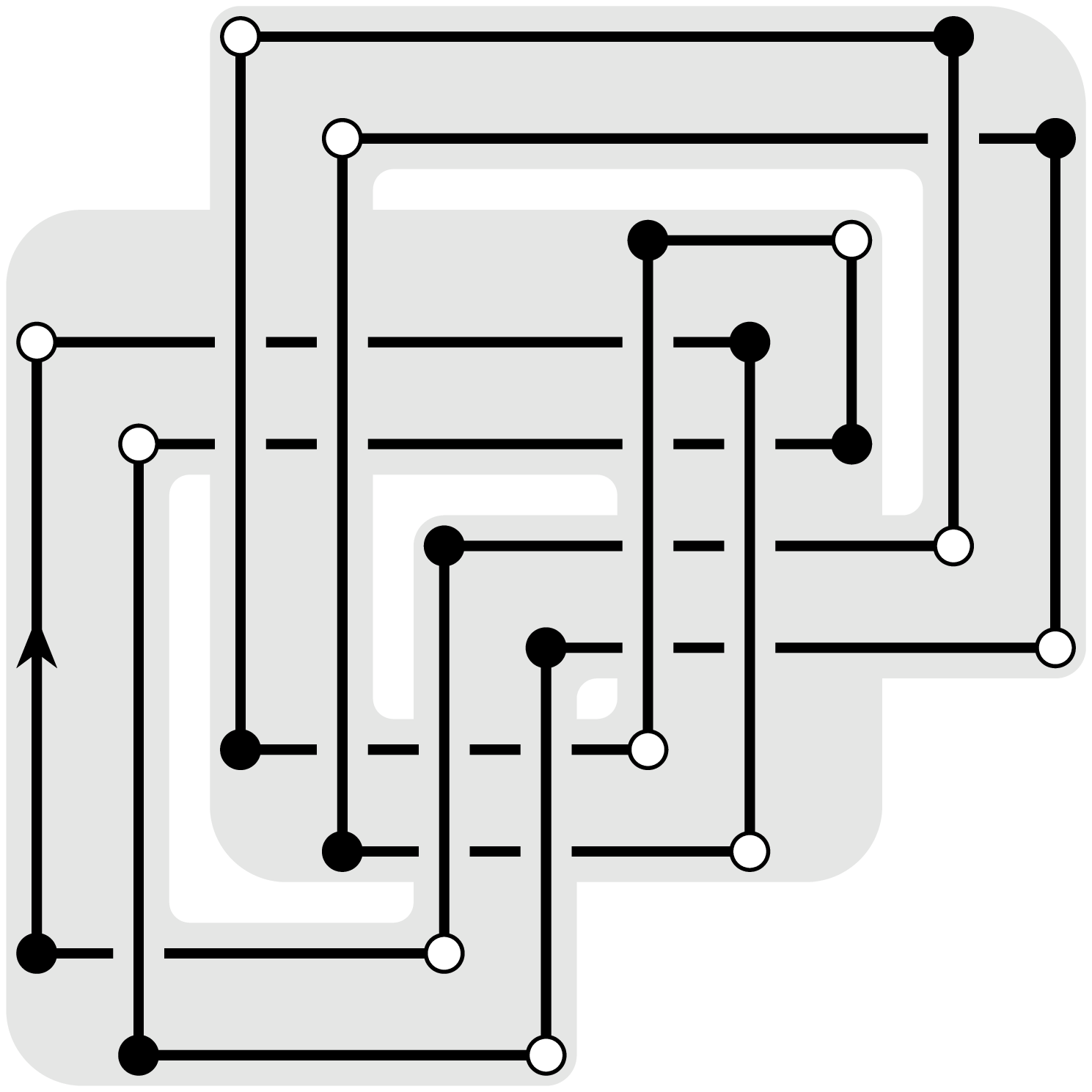}

\vskip-8mm
$$R_{\text{\Sun}}$$

\begin{tabular}{ccc}
\includegraphics[width=150pt]{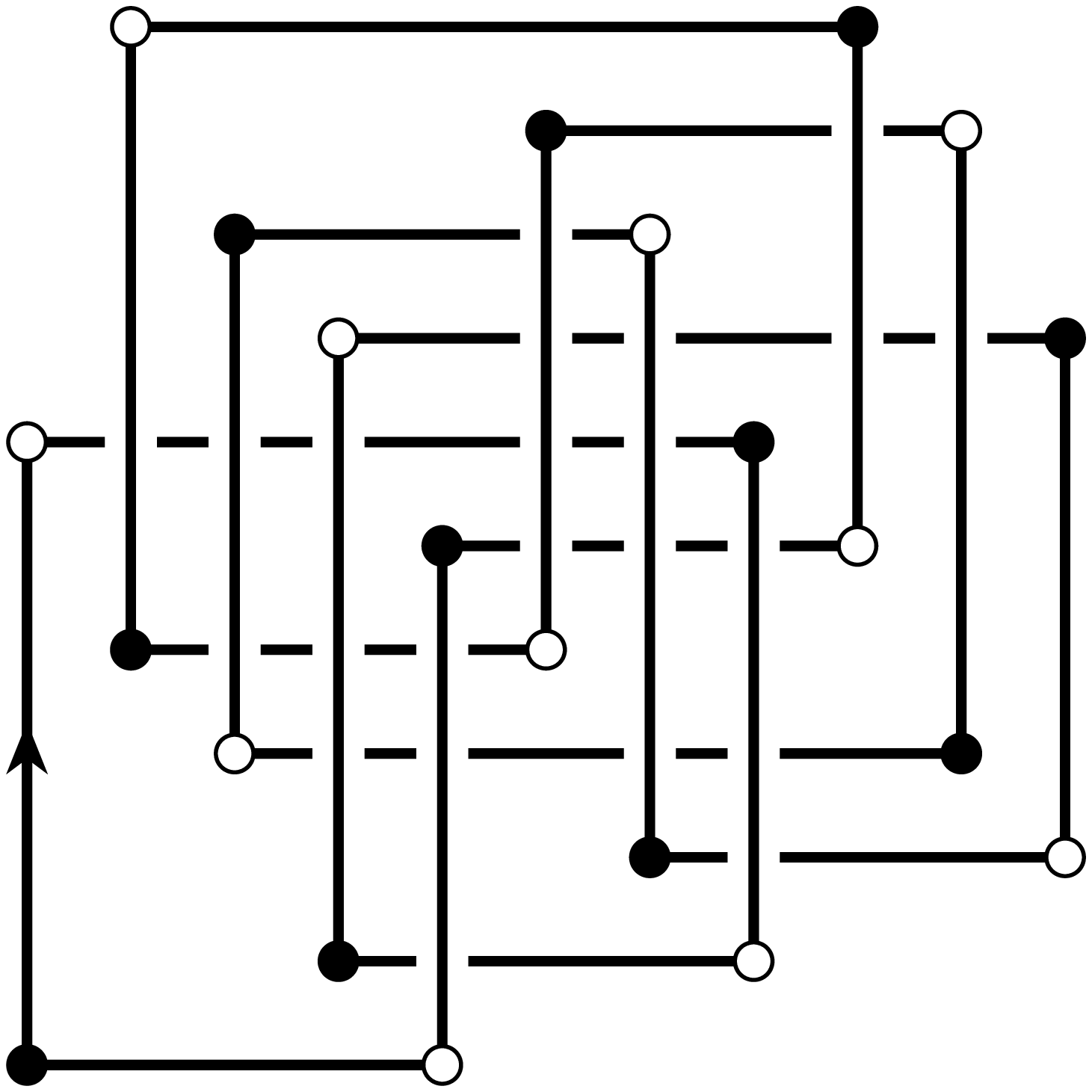}
&\hspace{1cm}&
\includegraphics[width=150pt]{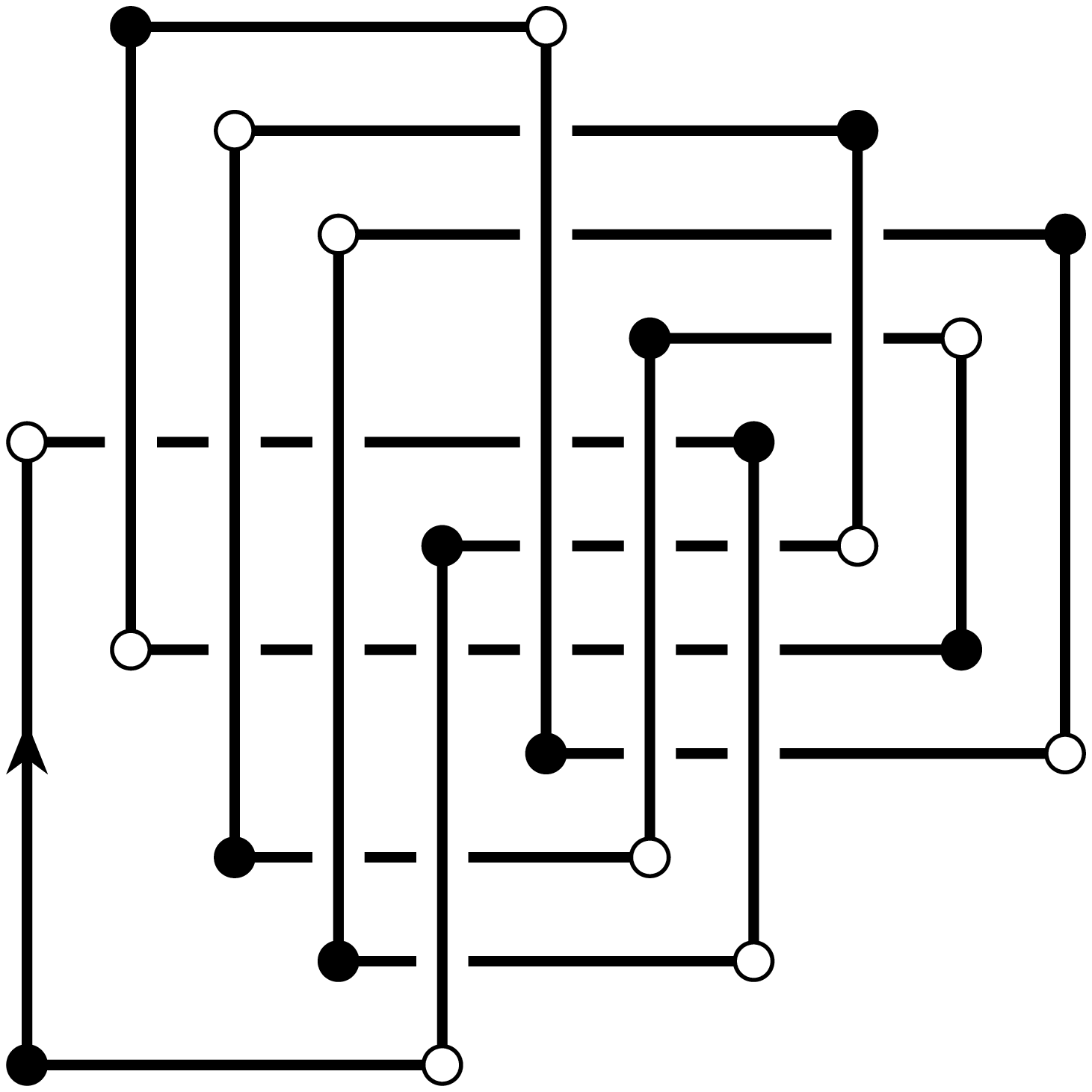}
\\[-3mm]
$R_{\text{\Cloud}}$&&$R_{\text{\SunCloud}}$
\\[5mm]
\includegraphics[width=150pt]{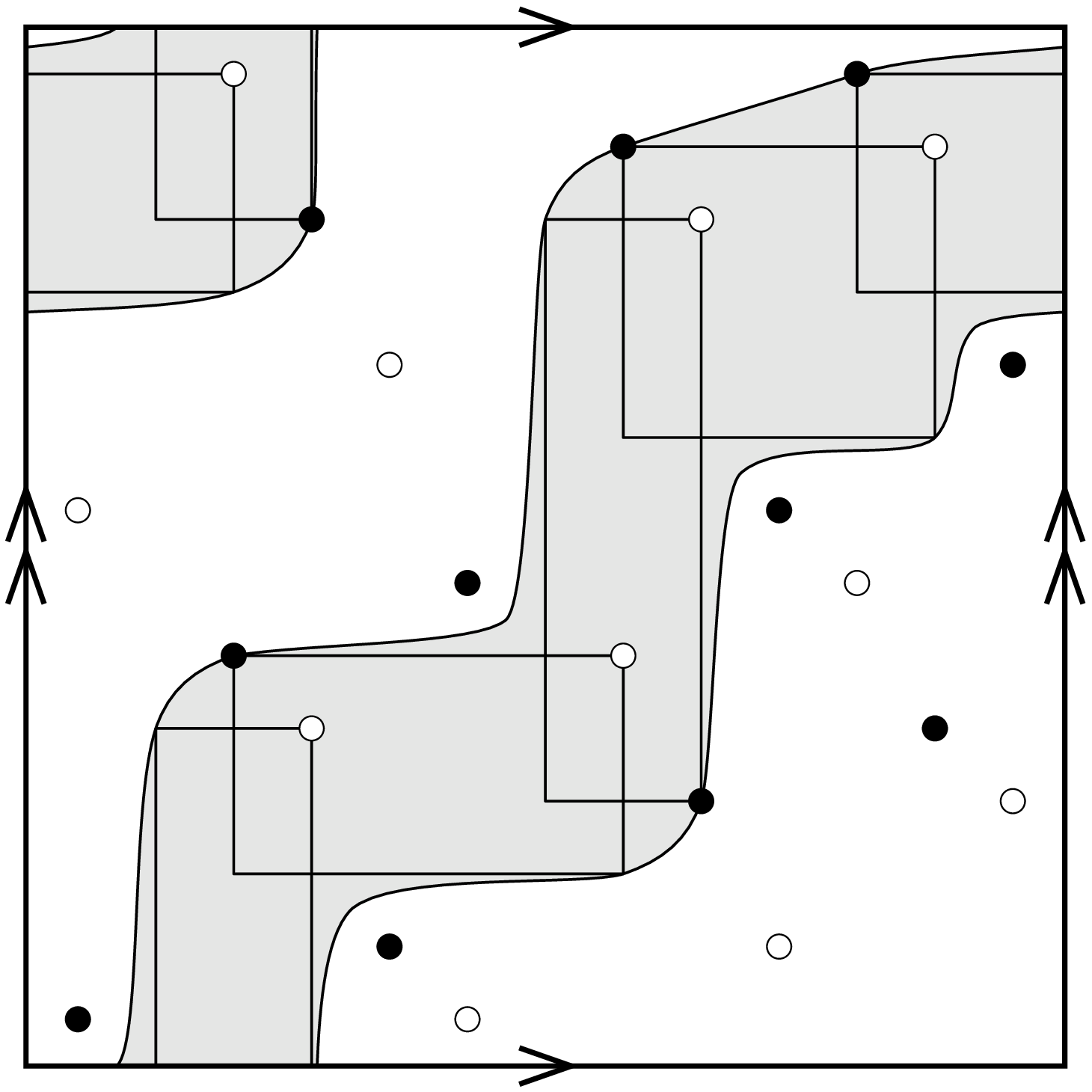}
&&
\includegraphics[width=150pt]{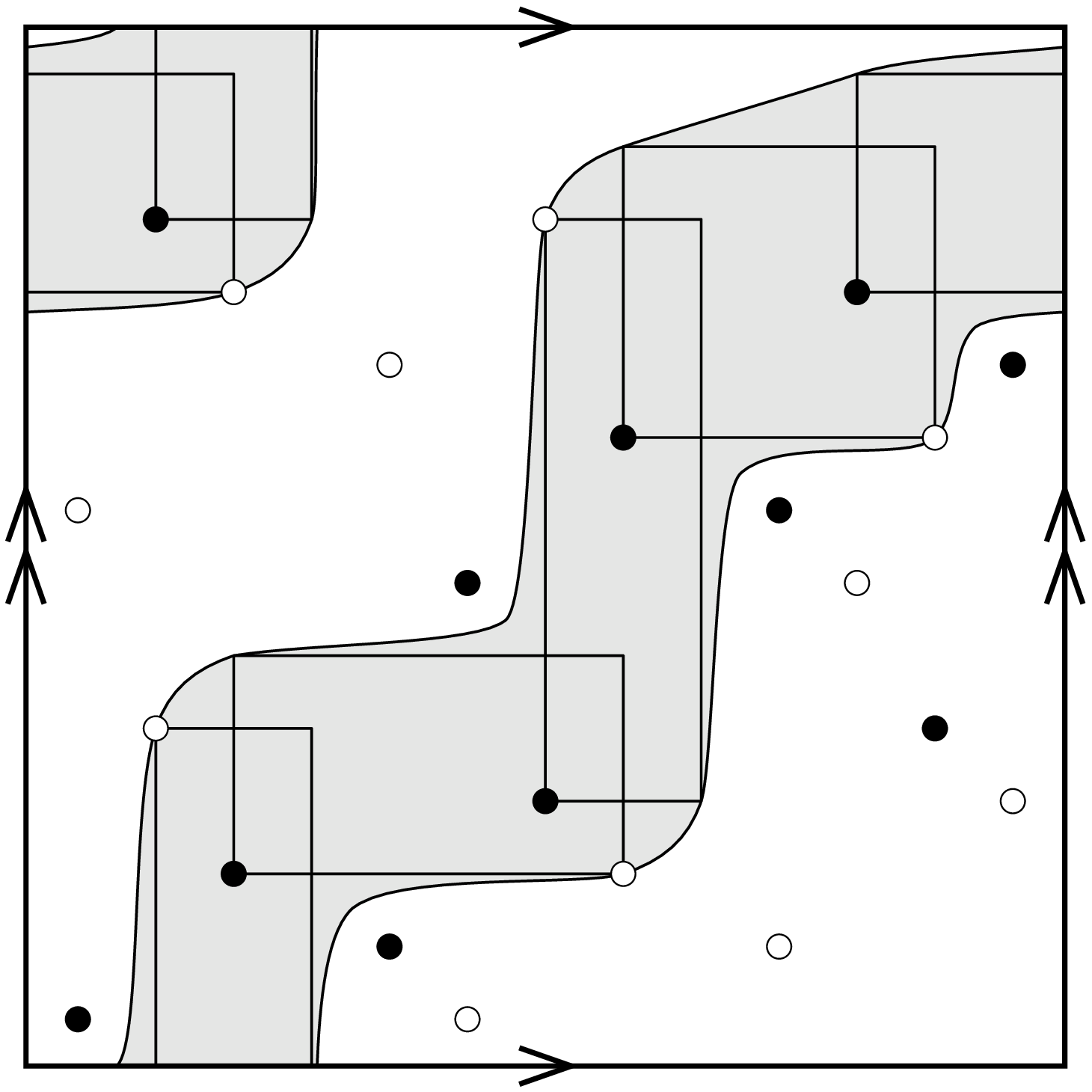}
\end{tabular}
\caption{The transitions~$R_{\text{\Cloud}}\mapsto R_{\text{\SunCloud}}$ and~$R_{\text{\SunCloud}}\mapsto R_{\text{\Cloud}}$ are (combinatorially) a $\swarrow$-flype and
a $\nearrow$-flype, respectively}\label{exam-fig}
\end{figure}
It is then two elementary moves preserving the number of edges
(exchange moves) to obtain~$R_{\text{\Sun}}$ from~$R_{\text{\SunCloud}}$ (a shift one step up is also in order).

\section{Proof of Theorem~\ref{main-th}}\label{proof-sec}
\subsection{Preparations}
We keep the notation and the settings from Section~\ref{def-sec}.
In particular, we use
the bijection~$v\mapsto\overline v$ from~$A\setminus\partial A$ to itself
and extend it to the whole of~$A$ by continuity. Namely, for~$v\in\partial_1A$ (respectively, $v\in\partial_2A$),
the point~$\overline v$ is defined by the condition that a connected
component of the intersection of some meridian (respectively, longitude)
with~$A$ has the form~$\{\theta_0\}\times[\varphi_1;\varphi_2]$
with~$(\theta_0,\varphi_1)=v$  and~$(\theta_0,\varphi_2)=\overline v$
(respectively, $[\theta_1;\theta_2]\times\{\varphi_0\}$ with~$(\theta_1,\varphi_0)=v$
and~$(\theta_2,\varphi_0)=\overline v$). If~$v\in\partial A$,
then the notation~$\sigma_{r_v}$ refers to the identically zero function on~$\mathbb T^2$.

We assume that
$R\mapsto R'$ is a $\nearrow$-multiflype based an annulus~$A\subset\mathbb T^2$.

We say that an elementary move~$R_1\mapsto R_2$ is performed \emph{inside~$A$}
if~$\sigma_{R_1}-\sigma_{R_2}=\pm\sigma_r$, where~$r$ is a rectangle
contained in~$A$ such that~$r\cap R_1\subset V(r)$, where by~$V(r)$ we denote the set of vertices of~$r$.

\begin{lemm}\label{ind-step-lem}
Let~$R\mapsto R_1$ be an elementary move performed inside~$A$. Suppose that~$A$
is still suitable for defining a $\nearrow$-multiflype on~$R_1$.
Let~$R_1\mapsto R_1'$ be this $\nearrow$-multiflype.
Then~$R_1\mapsto R_1'$ is an elementary move performed inside~$A$.
\end{lemm}

\begin{proof}
Let~$r\subset A$ be  a rectangle such that~$\sigma_{R}-\sigma_{R_1}=\epsilon\sigma_r$ and
$r\cap R\subset V(r)$, where~$\epsilon=\pm1$,
and let~$v_1,v_2,v_3,v_4$ be the vertices of~$r$ numbered counterclockwise
with~$v_1$ being the bottom left vertex.

Equality~\eqref{sigma-flype-eq} can be rewritten as
$$\sigma_{R'}=\sigma_R-\sum_{v\in A}\sigma_R(v)\sigma_{r_v},$$
since there are only finitely many points at which~$\sigma_R$
does not vanish, and for~$v\in\partial A$ we put~$\sigma_{r_v}\equiv0$.
Similarly, we have
$$\sigma_{R_1'}=\sigma_{R_1}-\sum_{v\in A}\sigma_{R_1}(v)\sigma_{r_v},$$
and hence,
$$\sigma_{R'}-\sigma_{R'_1}=\epsilon\Bigl(\sigma_r-\sum_{v\in A}\sigma_r(v)\sigma_{r_v}\Bigr)
=\epsilon\Bigl(\sigma_r-\sum_{v\in V(r)}\sigma_r(v)\sigma_{r_v}\Bigr).$$
One can verify (consult Figure~\ref{r-bar-fig}) that, whichever rectangle~$r\subset A$ is,
\begin{figure}[ht]
\includegraphics[width=150pt]{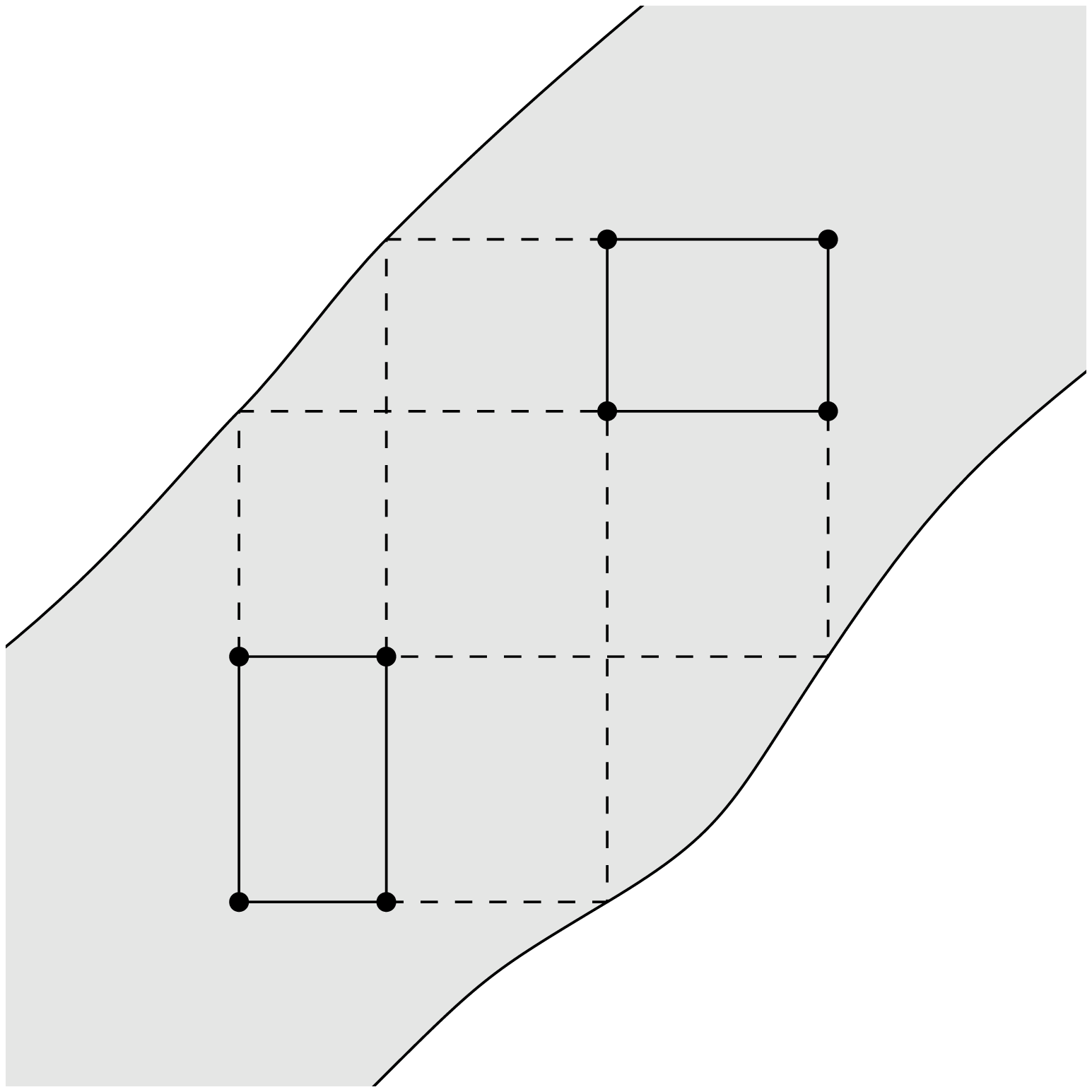}
\put(-108,42){$r$}\put(-54,102){$\overline r$}
\put(-127,20){$v_1$}\put(-95,20){$v_2$}
\put(-127,63){$v_4$}\put(-95,63){$v_3$}
\put(-76,84){$\overline v_1$}\put(-35,84){$\overline v_4$}
\put(-76,120){$\overline v_2$}\put(-35,120){$\overline v_3$}
\put(-145,5){$A$}
\caption{The rectangle~$\overline r$}\label{r-bar-fig}
\end{figure}
the following identity holds
$$\sigma_r-\sum_{v\in V(r)}\sigma_r(v)\sigma_{r_v}=-\sigma_{\overline r},$$
where~$\overline r=\{\overline u:u\in r\}$ is also a rectangle, and the vertices of~$\overline r$ listed
clockwise are~$\overline v_1$, $\overline v_2$, $\overline v_3$, $\overline v_4$. Thus, we have
$$\sigma_{R_1'}=\sigma_{R_1}+\epsilon\sigma_{\overline r}.$$

We have seen in the proof of Proposition~\ref{well-defined-prop} the following:
$$R'\cap(A\setminus\partial A)=\{\overline v:v\in R\cap(A\setminus\partial A)\}.$$
Since
$r\setminus V(r)\subset A\setminus\partial A$ and
$r\cap R\subset V(r)$,
we have
$\overline r\cap R'\subset V(\overline r)$.

To ensure that~$R'\mapsto R_1'$ is an elementary move it remains to verify that~$R'\cap V(\overline r)$
consists of exactly one, two, or three successive vertices of~$\overline r$.
This is equivalent to saying that there are two vertices of~$\overline r$ opposite to one another
and such that exactly one of them belongs to~$R'$.

If~$r\subset A\setminus\partial A$ then~$\overline v_i\in R'$ if and only if~$v_i\in R$.
In this case, $R'\cap V(\overline r)$ consists of exactly one, two, or three successive vertices of~$\overline r$,
since the same is true for~$R\cap V(r)$ and~$r$ by assumption.

The vertices~$v_1$ and~$v_3$ always lie in~$A\setminus\partial A$,
hence, if~$R\cap\{v_1,v_3\}=\{v_1\}\text{ or }\{v_3\}$,
then~$R'\cap\{\overline v_1,\overline v_3\}=\{\overline v_1\}\text{ or }\{\overline v_3\}$, so,
the required condition on the intersection~$R'\cap V(\overline r)$ holds true.

We are left with the cases when~$R\cap\{v_1,v_3\}=\varnothing\text{ or }\{v_1,v_3\}$,
and~$r\cap\partial A\ne\varnothing$. We may assume without loss of generality
that~$R\cap r=\{v_1,v_2,v_3\}$,
since the other remaining cases are obtained from this one by exchanging~$R$ with~$R_1$
and/or $\theta$ with~$\varphi$.

In this case, $v_4$ is a crossing of~$R$, therefore, by assumption, $v_4\notin\partial A$. The only
non-trivial option for~$r\cap\partial A$ is~$\{v_2\}$. It is a direct check that,
in this case, $R'\cap V(\overline r)=\{\overline v_1,\overline v_2,\overline v_3\}$. This completes
the proof of the lemma.\end{proof}

For any point~$v\in A\setminus\partial A$, denote by~$r^v$ the rectangle~$r_u$ with~$\overline u=v$.
Denote also by~$\Delta^+_v$ (respectively, $\Delta^-_v$) the closure of the connected component
of~$A\setminus(r_v\cup r^v)$ having empty intersection with~$\partial_1A$ (respectively, $\partial_2A$),
and by~$\Omega_v$ the union~$\Delta^+_v\cup\Delta^-_v\cup r_v$ (see Figure~\ref{omega-fig}).
\begin{figure}[ht]
\includegraphics[width=150pt]{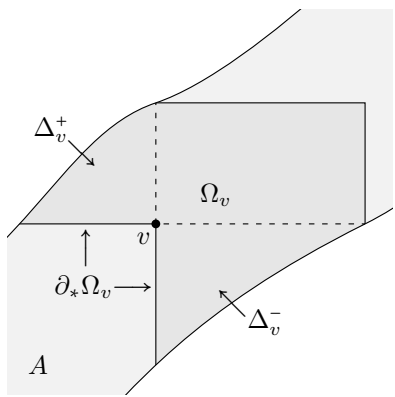}\put(-75,75){$\Omega_v$}
\put(-140,10){$A$}\put(-99,59){$v$}
\put(-127,97){\rotatebox{-45}{$\longrightarrow$}}\put(-138,100){$\Delta^+_v$}
\put(-70,40){\rotatebox{-45}{$\longleftarrow$}}\put(-57,27){$\Delta^-_v$}
\put(-120,50){\rotatebox{90}{$\longrightarrow$}}
\put(-108,40){$\longrightarrow$}\put(-130,40){$\partial_*\Omega_v$}
\caption{The domain~$\Omega_v$}\label{omega-fig}
\end{figure}
By~$\partial_*\Omega_v$ we denote the following part of the boundary of~$\Omega_v$:
$$\partial_*\Omega_v=(\partial\Omega_v\setminus\partial A)\cap r^v.$$
One can see that~$u\in\partial_*\Omega_v$ is equivalent to~$\overline u\in\partial_*\Omega_{\overline v}$,
and~$\partial_*\Omega_{\overline v}\subset\partial\Omega_v$.

Now choose a point~$u_0\in A\setminus\partial A$ such that neither~$u_0$ nor~$\overline u_0$
belongs to a meridian or a longitude containing an edge of~$R$.
Denote~$\overline u_0$ by~$u_1$.

The proof of Theorem~\ref{main-th} is by induction in the number of vertices of~$R$
contained in~$\Omega_{u_0}\setminus\partial A$.

\subsection{The induction base}
Suppose that~$R\cap(\Omega_{u_0}\setminus\partial A)=\varnothing$.
Pick a smooth parametrized path~$t\mapsto u_t$, $t\in[0;1]$, starting at~$u_0$ and ending at~$u_1$ and such that:
\begin{enumerate}
\item
$u_t\in A\setminus(\partial A\cup\Omega_{u_0})$ for all~$t\in (0;1)$;
\item
$u_t$ avoids crossings and vertices of~$R$;
\item
$d\theta(u_t)/dt<0$ and~$d\varphi(u_t)/dt<0$ for all~$t\in[0;1]$.
\end{enumerate}
Observe that we also have~$d\theta(\overline u_t)/dt<0$ and~$d\varphi(\overline u_t)/dt<0$ for all~$t\in[0;1]$.

For brevity, denote~$\Omega_{u_t}$ by~$\Omega_t$.
For~$0<t'<t''<1$, denote also by~$\Omega_{t',t''}$ the union~$\bigcup_{t\in(t';t'']}\Omega_t$.
Clearly, we have
$$\bigcup_{t\in(0,1)}\partial_*\Omega_t=A\setminus(\partial A\cup\Omega_0),$$
hence, all points from~$R\cap(A\setminus\partial A)$ are contained in the union~$\bigcup_{t\in(0,1)}\partial_*\Omega_t$.

Let~$t_1<t_2<\ldots<t_m$ be all moments~$t\in(0;1)$ at which a vertex of~$R$ appears
on~$\partial_*\Omega_{u_t}$. Put~$R_0=R$, and define oriented rectangular diagrams of a link~$R_1,R_2,\ldots,R_m$
as follows:
\begin{equation}\label{ri-eq}
\sigma_{R_i}=\sigma_R-\sum_{v\in\Omega_{0,t_i}}\sigma_R(v)\sigma_{r_v},\quad i=1,\ldots,m.
\end{equation}
By construction, we have~$R_m=R'$.

Now we claim that~$R_{i-1}\mapsto R_i$ is either an elementary move or a composition of two
elementary moves for any~$i=1,\ldots,m$. Indeed, according to~\eqref{ri-eq},
the intersection~$R_{i-1}\cap(A\setminus\partial A)$ is obtained from~$R\cap(A\setminus\partial A)$
by replacing each vertex~$v\in R\cap(\Omega_{0,t_{i-1}}\setminus\partial A)$
with~$\overline v$. If~$v\in R\cap(\Omega_{0,t_{i-1}}\setminus\partial A)$,
then~$v\in\partial_*\Omega_{t_j}$ for some~$j=1,\ldots,i-1$,
which implies $\overline v\in\partial_*\Omega_{\overline u_{t_j}}$.
The union~$\bigcup_{j=0}^{i-1}\partial_*\Omega_{\overline u_{t_j}}$ is disjoint
from~$\Omega_{t_i}$. Therefore, the only intersection of~$\Omega_{t_i}$ with~$R_{i-1}$
consists of vertices of~$R$ lying at~$\partial_*\Omega_{t_i}$.

Since, by construction, $u_{t_i}$ is not a vertex or a crossing of~$R$,
there are at most two vertices of~$R$ in~$\partial_*\Omega_{t_i}$.

Suppose that there is a single vertex~$v$, say, in~$R\cap\partial_*\Omega_{t_i}$.
The rectangle~$r_v$ is a subset of~$\Omega_{t_i}$, therefore,
$$R_{i-1}\cap r_v\subset R_{i-1}\cap\Omega_{t_i}\subset\partial_*\Omega_{t_i}\cup\partial A.$$
We also have
$$r_v\cap(\partial_*\Omega_{t_i}\cup\partial A)\subset V(r_v),$$
which implies~$R_{i-1}\cap r_v\subset V(r_v)$.
We also have~$\overline v\notin\partial_*\Omega_{t_i}\cup\partial A$. Thus~$v\in R_{i-1}\cap r_v$
and~$\overline v\notin R_{i-1}\cap r_v$. This implies that~$R_{i-1}\mapsto R_i$
is an elementary move.

Now suppose that~$R\cap\partial_*\Omega_{t_i}$ consists of two vertices of~$R$.
Denote the one which closer to~$u_{t_i}$ (in the Euclidean metric restricted to~$\Omega_{t_i}$)
by~$v_1$, and the other one by~$v_2$. Define~$R_{i-1}'$ by
$$\sigma_{R_{i-1}'}=\sigma_{R_{i-1}}-\sigma_{R_{i-1}}(v_1)\sigma_{r_{v_1}}.$$
Then the transition~$R_{i-1}\mapsto R_{i-1}'$ is an elementary move for the same reason
as in the previous case.

To see that~$R_{i-1}'\mapsto R_i$ is an elementary move we note that~$R_{i-1}\cap(r_{v_2}\setminus V(r_{v_2}))$
contains only the vertex~$v_1$, which is no longer present in~$R_{i-1}'$.
It is replaced by~$\overline v_1$, which is outside of~$r_{v_2}$ (see Figure~\ref{v1v2-fig}).
\begin{figure}[ht]
\includegraphics[width=150pt]{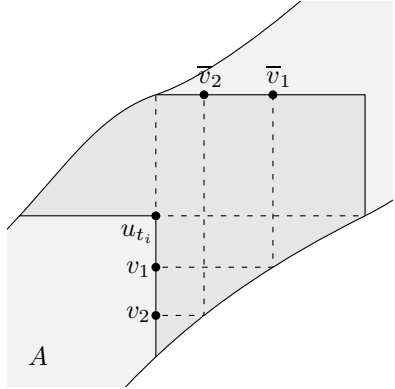}
\put(-140,10){$A$}\put(-105,59){$u_{t_i}$}
\put(-103,45){$v_1$}\put(-103,28){$v_2$}
\put(-50,117){$\overline v_1$}
\put(-76,117){$\overline v_2$}
\caption{The case when two vertices of~$R_{i-1}$ appear
on~$\partial_*\Omega_{t_i}$}\label{v1v2-fig}
\end{figure}

Thus, we have found a sequence of elementary moves producing~$R'$ from~$R$ in the case
when~$R\cap(\Omega_0\setminus\partial A)=\varnothing$.

\subsection{The induction step}
Suppose that~$|R\cap(\Omega_0\setminus\partial A)|=m>0$ and
the theorem is proved in the case when~$|R\cap(\Omega_0\setminus\partial A)|<m$.
We are going to find an elementary move~$R\mapsto R_1$ performed inside~$A$
such that~$A$ is still suitable for defining a $\nearrow$-multiflype on~$R_1$ (possibly after
a small modification of~$A$ not affecting the multiflype~$R\mapsto R'$),
and~$|R\cap(\Omega_0\setminus\partial A)|=m-1$. The induction step will then follow from
Lemma~\ref{ind-step-lem}.

Denote: $(\theta_0,\varphi_0)=u_0$, $(\theta_1,\varphi_1)=u_1=\overline u_0$.
Let~$v=(\theta_2,\varphi_2)$ be the closest to~$u_1$ point in~$R\cap(\Omega_0\setminus\partial A)$
(if there are more than one such point choose any of them).
There are the following three cases to consider.

\medskip\noindent\emph{Case 1}: $v\in r_{u_0}\setminus\partial r_{u_0}$. For~$\varepsilon>0$,
define~$r(\varepsilon)$ to be the rectangle~$[\theta_2;\theta_1+\varepsilon]\times[\varphi_2;\varphi_1+\varepsilon]$
(see the left picture in Figure~\ref{repsilon-fig}).
\begin{figure}[ht]
\includegraphics[width=150pt]{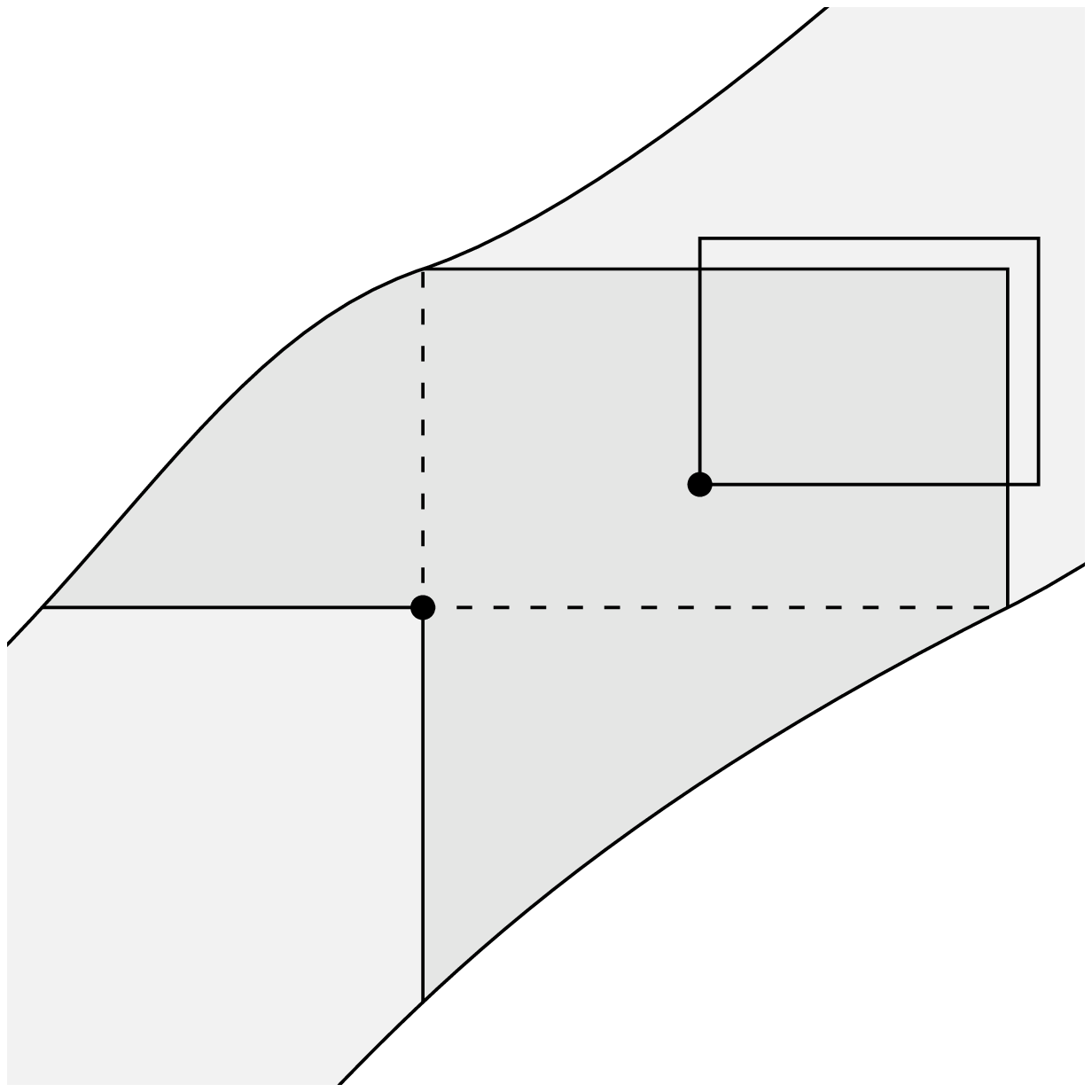}
\put(-140,10){$A$}\put(-103,59){$u_0$}\put(-60,75){$v$}\put(-38,97){$r(\varepsilon)$}
\hskip2cm
\includegraphics[width=150pt]{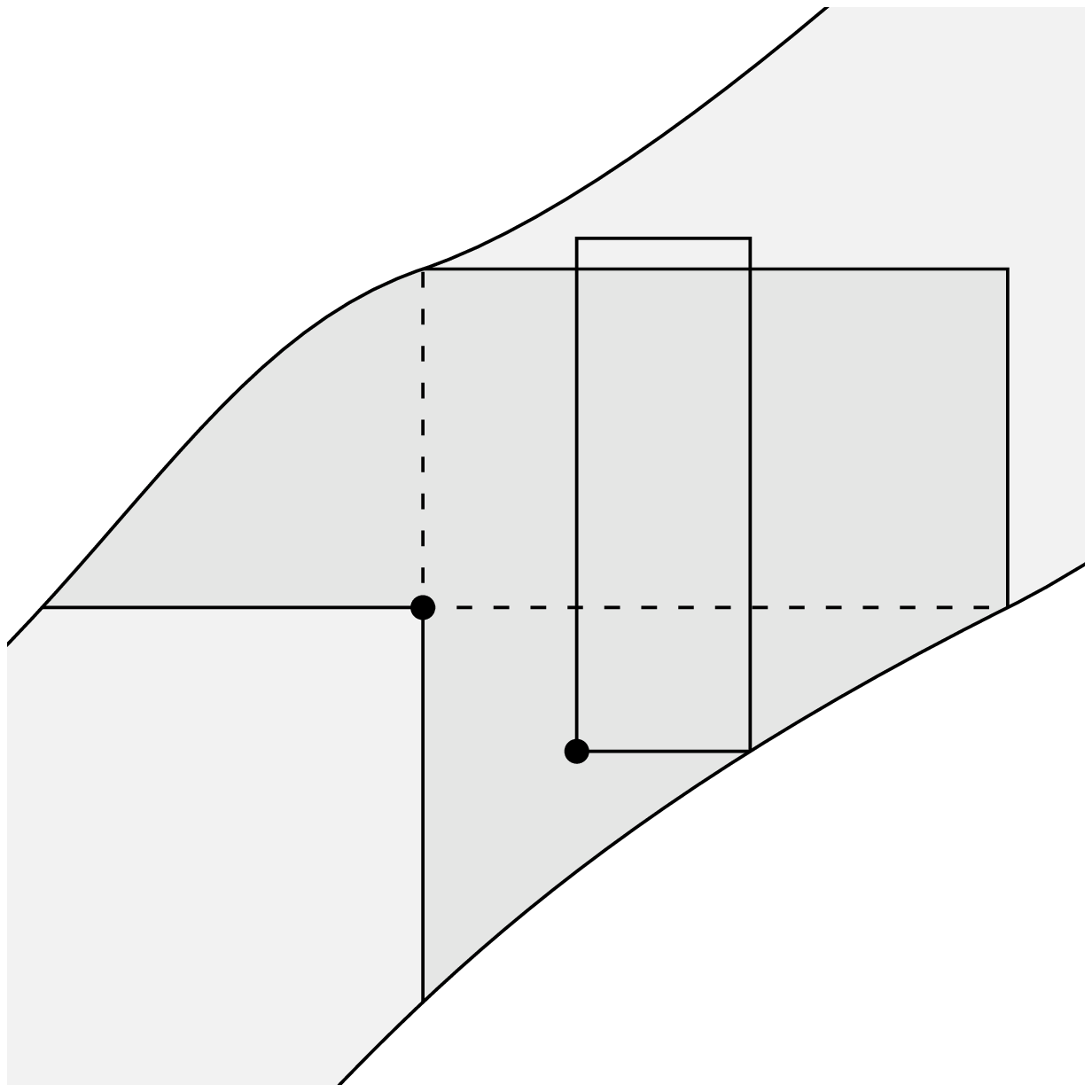}
\put(-140,10){$A$}\put(-103,59){$u_0$}\put(-77,39){$v$}\put(-67,80){$r(\varepsilon)$}
\caption{The rectangle~$r(\varepsilon)$}\label{repsilon-fig}
\end{figure}
For small enough~$\varepsilon$, the following conditions hold:
\begin{enumerate}
\item
the rectangle~$r(\varepsilon)$ is contained in~$A$;
\item
$R\cap r(\varepsilon)=v$;
\item
the meridian~$m_{\theta_1+\varepsilon}$ and the longitude~$\ell_{\varphi_1+\varepsilon}$
are disjoint from~$R$.
\end{enumerate}
Therefore, there is an elementary move~$R\mapsto R_1$ performed inside~$A$ such
that~$\sigma_{R_1}=\sigma_R-\sigma_R(v)\sigma_{r(\varepsilon)}$.
We clearly have~$|R_1\cap(\Omega_0\setminus\partial A)|=m-1$ as~$v$
has been replaced by three vertices outside of~$\Omega_0$.

By choosing~$\varepsilon$ small enough we can also ensure that~$A$ is still suitable to define
a $\nearrow$-multiflype on~$R_1$. Indeed, due to the nature of the conditions imposed on~$A$,
there are only finitely many~$\varepsilon$ for which those conditions are violated.

\medskip\noindent\emph{Case 2}: $v\in\Delta^-_{u_0}$. Denote~$\theta_3=p_\theta(\overline v)$.
We define~$r(\varepsilon)$ to be the rectangle~$[\theta_2;\theta_3]\times[\varphi_2;\varphi_1+\varepsilon]$
(see the right picture in Figure~\ref{repsilon-fig})
and proceed as in the previous case. A minor subtlety occurs only when~$(\theta_3,\varphi_2)$
is not a vertex of~$R$, in which case the diagram~$R_1$ is forced to have an edge at
the new meridian~$m_{\theta_3}$, which does not depend on~$\varepsilon$.
This may result in failing of
the last condition imposed on~$A$ for defining a $\nearrow$-multiflype on~$R_1$.

However, this is easily resolved by a small perturbation of~$\partial A$
near the intersections with~$m_{\theta_3}$ other than~$(\theta_3,\varphi_2)$.
Such perturbations do not affect the flype~$R\mapsto R'$, since these points
are not contained in any longitude or meridian passing through a vertex of~$R$.

\medskip\noindent\emph{Case 3}: $v\in\Delta^+_{u_0}$.
This case is symmetric to the previous one and left to the reader.

\medskip
The proof of Theorem~\ref{main-th} is now complete.

\section{Concluding remarks}
By \emph{a $\diagup$-multiflype} (respectively, \emph{a $\diagdown$-multiflype})
we call any $\nearrow$- or $\swarrow$-multiflype (respectively, $\nwarrow$- or $\searrow$-multiflype).

The proof of Theorem~\ref{main-th} given above provides an algorithm for
decomposing any multiflype into a sequence of elementary moves.
By following the lines of the proof one can see that the decomposition of
a $\diagup$-multiflype consists of elementary
moves that are particular cases of $\diagup$-multiflypes.
Similarly for $\diagdown$-multiflypes.

Any exchange move (or commutation) of rectangular diagrams of links can be simultaneously
viewed as a $\diagup$-multiflype and a $\diagdown$-multiflype.

Stabilizations and destabilization which are $\diagup$-multiflypes are exactly those that
are called type~I (de)stabilization in~\cite{DyPr}. In the terminology of~\cite{OST},
these are (de)stabilizations of types~{\sl X:NE}, {\sl X:SW}, {\sl O:NE}, and~{\sl O:SW}.
Similarly, (de)stabilizations which are $\diagdown$-multiflypes
are those that are of type~II in~\cite{DyPr} and of types~{\sl X:NW}, {\sl X:SE}, {\sl O:NW}, and~{\sl O:SE}
in~\cite{OST}.

With every rectangular diagram of a link~$R$, one associates two Legendrian link types,
one with respect to the standard contact structure~$\xi_{\mathrm{st}}$, and the other
with respect to the mirror image of~$\xi_{\mathrm{st}}$ (see~\cite{DyPr,OST}). We denote them
here by~$\mathscr L_\diagup(R)$ and~$\mathscr L_\diagdown(R)$, respectively.

Due to the remark above the relation between rectangular diagrams of links
and Legendrian links, which is explained in~\cite{DyPr,OST}, can be summarized as follows.

\begin{coro}\label{leg-coro}
Let~$R_1$ and~$R_2$ be oriented rectangular diagrams of links.
We have~$\mathscr L_\diagup(R_1)=\mathscr L_\diagup(R_2)$
\emph(respectively, $\mathscr L_\diagdown(R_1)=\mathscr L_\diagdown(R_2)$\emph)
if and only if~$R_1$ and~$R_2$ are related by a sequence of $\diagup$-multiflypes
\emph(respectively, $\diagdown$-multiflypes\emph).
\end{coro}

In~\cite{pras}, flypes of rectangular diagrams of links are generalized to the case of rectangular
diagrams of graphs. One can similarly generalize multiflypes and
Theorem~\ref{main-th} for general graphs, as well as Corollary~\ref{leg-coro} for Legendrian graphs,
and the proof will need no essential change.

\end{document}